\documentclass[a4paper,11pt]{article}

\usepackage{amsfonts,qtree,stmaryrd,textcomp}
\usepackage{pifont,amssymb,amsmath,amsthm,tabularx,graphicx}
\usepackage{tree-dvips,pictexwd,dcpic}
\usepackage{bbm}

\usepackage{stmaryrd}

\usepackage[T1]{fontenc}
\usepackage[latin1]{inputenc}
\usepackage[text={17cm,26cm},centering]{geometry}

\usepackage{etex}
\usepackage{fancyhdr}
\usepackage{graphicx}
\usepackage{enumitem}
\usepackage{amsmath}
\usepackage[bbgreekl]{mathbbol}
\usepackage{amssymb}
\usepackage{amsthm}
\usepackage[all]{xy}

\usepackage{epigraph}
\RequirePackage{bussproofs}

\usepackage{framed}
\usepackage{mathtools}
\usepackage{url}
\usepackage{proof}
\usepackage{xspace}
\usepackage{listings}
\usepackage{wrapfig}
\usepackage{tikz}
\usepackage{tikz-cd}

\usepackage{relsize}

\usetikzlibrary{positioning}
\usetikzlibrary{shapes}

\usetikzlibrary{arrows}
\usetikzlibrary{calc}
\usepackage{tikz-3dplot}
\usetikzlibrary{decorations.pathmorphing}

\usepackage{amsfonts,amssymb,amsmath,qtree,stmaryrd,textcomp, amsthm}

\newcommand{\C}[1]{\mathcal{#1}}

\usepackage{ mathrsfs }

\newtheorem{theorem}{Theorem}[section]
\newtheorem{proposition}[theorem]{Proposition}
\newtheorem{lemma}[theorem]{Lemma}
\newtheorem{corollary}[theorem]{Corollary}

\newtheorem{example}[theorem]{Example}

\newtheorem{definition}[theorem]{Definition}

\newcommand{\Uni}{{\mathcal U}}

\newcommand{\Bii}{\C B}

\newcommand{\one}{{\overline{1}}}

\newcommand{\id}{\mathrm{id}}

\newcommand{\Typ}{\mathrm{Typ}}

\newcommand{\pr}{\textnormal{\texttt{pr}}}

\newcommand{\Uii}{\mathcal{U}}

\newcommand{\xyA}{x, y : A}

\newcommand{\EA}{=_{A}}

\newcommand{\Exy}{x \EA y}

\newcommand{\Cii}{\mathcal{C}}

\newcommand{\ap}{\textnormal{\texttt{ap}}}

\newcommand{\fAB}{f: A \rightarrow B}

\newcommand{\isequiv}{\textnormal{\texttt{isequiv}}}

\newcommand{\eqv}{\textnormal{\texttt{eqv}}}

\newcommand{\Aii}{\mathcal{A}}
\newcommand{\Idtoeqv}{\textnormal{\texttt{IdtoEqv}}}

\newcommand{\idtoeqv}{\textnormal{\texttt{idtoEqv}}}

\newcommand{\Ua}{\textnormal{\texttt{Ua}}}

\newcommand{\funext}{\textnormal{\texttt{funext}}}

\newcommand{\ua}{\textnormal{\texttt{ua}}}

\newcommand{\epair}{^{=}\textnormal{\texttt{pair}}}
\newcommand{\paire}{\textnormal{\texttt{pair}}^{=}}

\newcommand{\Typfun}{\mathrm{\texttt{Typfun}}}

\newcommand{\refl}{\textnormal{\texttt{refl}}}

\newcommand{\isProp}{\textnormal{\texttt{isProp}}}
\newcommand{\isSet}{\textnormal{\texttt{isSet}}}

\newcommand{\At}{\Aii^{t}}

\newcommand{\happly}{\textnormal{\texttt{happly}}}

\newcommand{\Fun}{\mathrm{Fun}}

\begin{document}

\date{}

\title{\textbf{Univalent typoids}}

\author{Iosif Petrakis\\	
Mathematics Institute, Ludwig-Maximilians Universit\"{a}t M\"{u}nchen\\
petrakis@math.lmu.de}  

%






\maketitle

\begin{abstract}
A typoid is a type equipped with an equivalence relation, such that the terms of
equivalence between the terms of the type satisfy certain conditions, with respect to 
a given equivalence relation between them, that generalise the properties of the equality terms.
The resulting weak 2-groupoid structure can be extended to every finite level.
The introduced notions of typoid and typoid function generalise the notions of setoid 
and setoid function. A univalent typoid is a typoid satisfying a 
general version of the univalence axiom. We prove some fundamental facts on univalent typoids, 
their product and exponential. As a corollary, we get an interpretation of propositional truncation
within the theory of typoids. The couple typoid and univalent typoid is a weak groupoid-analogue to 
the couple precategory and category in homotopy type theory.
\end{abstract}

\section{Introduction}
\label{sec:intro}

One of the key-features of Martin-L\"{o}f's intensional type theory (ITT) (see~\cite{ML75},~\cite{ML98})
is the use of two kinds of equality for the terms $a, b$ of a
type $A$. The definitional, or judgmental equality $a \equiv b$, expresses that $a$ and $b$
are by definition equal, while the propositional equality $a =_A b$, or simpler $a = b$, is a new type, 
and every term
$p : a =_A b$ can be understood as a proof that $a$ and $b$ are propositionally equal. Through the (rough)
homotopic interpretation of 
ITT\footnote{This was formulated first by Voevodsky in~\cite{Vo06}, Awodey and Warren in~\cite{AW09},
and was inspired by Hofmann and Streicher's groupoid interpretation of ITT in~\cite{HS98}.} and the development 
of homotopy type theory (HoTT),
``$p$ is a path from the point $a$ to the point $b$ in space $A$''. The passage from
proofs of equality to proofs of equivalence is trivial through the use of Martin-L\"of's 
$J$-rule\footnote{Within the homotopic interpretation of ITT this rule is also called ``path-induction''.}, 
the induction principle that corresponds to the inductive definition of the type family $=_A : A \to A \to \Uii$,
where $\Uii$ is a fixed universe such that $A : \Uii$.  The passage from proofs of equivalence to proofs 
of equality is non-trivial. One needs the axiom of function extensionality (FE) to generate terms of equality 
from terms of equivalence between functions, and Voevodsky's axiom of univalence (UA) to generate terms of equality 
from terms of equivalence between types in a fixed universe (see~\cite{HoTT13}).

The following question arises naturally. ``Is it possible to have a common framework for all instances of getting
terms of equality from terms of equivalence?''

To answer this question, we introduce the notions of typoid, a generalisation of the notion of setoid, and of 
univalent typoid. Another approach to this question form the notions of precategory and category in HoTT 
(see Chapter 9 in book-HoTT~\cite{HoTT13}). As we discuss in section~\ref{sec: concl}, the couple typoid 
and univalent typoid is a weak groupoid-analogue to the couple precategory and category within HoTT. 

A setoid is the interpretation of Bishop's notion of set\footnote{Bishop sets were introduced in~\cite{Bi67}. 
For a recent
reconstruction of their theory within Bishop-style constructive mathematics see~\cite{Pe19,Pe20}.} in ITT
(see~\cite{Pa05},~\cite{BC03},~\cite{CS07}). It consists of a type $A$ paired with an equivalence 
relation $\simeq_{A}: A \rightarrow A \rightarrow \Uii$ on $A$. The morphisms between setoids are 
the functions between 
their types that respect the corresponding equivalences. 
A typoid\footnote{A similar notion is found in~\cite{CS16} and~\cite{CCCS18}, where the additional 
structure of the product and coproduct of setoids is considered at level two.}is a type $A$ with an 
equivalence relation $\simeq_{\Aii}$, such that the termss of 
equivalence between the terms of $A$ satisfy certain conditions with respect to 
a given, one level higher, equivalence relation $\cong_{\Aii}$ between them.  These conditions generalise 
the properties of terms of equality between the terms of $A$. The resulting weak 2-groupoid-structure of 
a typoid\footnote{The more accurate term for the notion of typoid 
that is studied here is that of a $2$-typoid, which is avoided only for simplicity.}, together with its 
corresponding notion of morphism, can be extended to any
finite level. 

A univalent typoid is a typoid that satisfies a general version of Voevodsky's axiom of univalence. 
Roughly speaking, a univalent typoid is a typoid $A$ such that a term $e : x \simeq_{\Aii} y$ generates a 
term $p : x =_A y$. 
In the next sections we prove some first fundamental facts on univalent types, their product and exponential. 
As a corollary, we get an interpretation of propositional truncation within the theory of typoids. 

We work within the informal framework of univalent type theory (UTT) i.e., of ITT extended with UA, 
which is found in the book-HoTT~\cite{HoTT13}. All proofs not included here are omitted as straightforward.

\section{Typoids and typoid functions}
\label{sec: basic}

\begin{definition}\label{def: typ}
A structure $\Aii \equiv \big(A, \simeq_{\mathsmaller{\Aii}}, 
\eqv_{\mathsmaller{\Aii}}, \ast_{\mathsmaller{\Aii}}, ^{-1_{\mathsmaller{\Aii}}}, \cong_{\mathsmaller{\Aii}}\big)$ 
is called a $2$-\textit{typoid}, or simply a \textit{typoid}, if  
$A : \Uii$ and $\simeq_{\mathsmaller{\Aii}} : \prod_{x, y : A}\Uii$
is an equivalence relation on $A$ such that 
$$\eqv_{\mathsmaller{\Aii}} : \prod_{x : A} (x \simeq_{\mathsmaller{\Aii}} x),$$ 
$$\ast_{\mathsmaller{\Aii}} : \prod_{x, y, z : A}\prod_{e : x \simeq_{\mathsmaller{\Aii}} y}\prod_{d : y \simeq_{\mathsmaller{\Aii}} z}x \simeq_{\mathsmaller{\Aii}} z,$$
$$^{-1_{\mathsmaller{\Aii}}} : \prod_{x, y : A}\prod_{e : x \simeq_{\mathsmaller{\Aii}} y}y \simeq_{\mathsmaller{\Aii}} x$$
and $\cong_{\mathsmaller{\Aii}} : \prod_{x, y : A}\prod_{e, d : x \simeq_{\mathsmaller{\Aii}} y}\Uii$
such that 
$$\cong_{\mathsmaller{\Aii}}(x, y) : \prod_{e, d : x \simeq_{\mathsmaller{\Aii}} y}\Uii$$
is an equivalence relation on $x \simeq_{\mathsmaller{\Aii}} y$, 
for every $x, y : A$. For simplicity we write $\eqv_{x}$ instead of $\eqv_{\mathsmaller{\Aii}}(x)$. Moreover,
If $e, e_1, d_1 : x \simeq_{\mathsmaller{\Aii}} y$, $e_2, d_2 : y \simeq_{\mathsmaller{\Aii}} z$, and $e_3 : z \simeq_{\mathsmaller{\Aii}} w$, the following conditions are satisfied:\\[1mm]
$(\Typ_1)$ $(\eqv_{x} \ast_{\mathsmaller{\Aii}} e) \cong_{\mathsmaller{\Aii}} e$ and $(e \ast_{\mathsmaller{\Aii}} \eqv_{y}) \cong_{\mathsmaller{\Aii}} e$.\\[1mm]
$(\Typ_2)$ $(e \ast_{\mathsmaller{\Aii}} e^{-1_{\mathsmaller{\Aii}}}) \cong_{\mathsmaller{\Aii}} \eqv_{x}$ and $(e^{-1_{\mathsmaller{\Aii}}} \ast_{\mathsmaller{\Aii}} e)
\cong_{\mathsmaller{\Aii}} \eqv_{y}$.\\[1mm]
$(\Typ_3)$ $(e_{1} \ast_{\mathsmaller{\Aii}} e_{2}) \ast_{\mathsmaller{\Aii}} e_{3} \cong_{\mathsmaller{\Aii}} e_{1} \ast_{\mathsmaller{\Aii}} (e_{2} \ast_{\mathsmaller{\Aii}} e_{3})$.\\[1mm]
$(\Typ_4)$ If $e_{1} \cong_{\mathsmaller{\Aii}} d_{1}$ and $e_{2} \cong_{\mathsmaller{\Aii}} d_{2}$, then $(e_{1} 
\ast_{\mathsmaller{\Aii}} e_{2}) \cong_{\mathsmaller{\Aii}} (d_{1} \ast_{\mathsmaller{\Aii}} d_{2})$.\\[2mm] 
An $1$-\textit{typoid}, or a \textit{setoid}, is just a couple $(A, \simeq_{\mathsmaller{\Aii}})$.
\end{definition}

From now on, $\Aii$ and $\Bii$ denote typoids, i.e.,
$\Aii \equiv \big(A, \simeq_{\mathsmaller{\Aii}}, \eqv_{\mathsmaller{\Aii}}, \ast_{\mathsmaller{\Aii}}, ^{-1_{\mathsmaller{\Aii}}}, \cong_{\mathsmaller{\Aii}}\big)$ and $\Bii \equiv \big(B,
\simeq_{\mathsmaller{\Bii}}, \eqv_{\mathsmaller{\Bii}}, \ast_{\mathsmaller{\Bii}}, ^{-1_{\mathsmaller{\Bii}}}, \cong_{\mathsmaller{\Bii}}\big)$.  If the context is clear, we may omit the subscripts.

\begin{proposition}\label{prp: inv1} Let $\Aii$ be a typoid, $x, y : A$, and $e, d : x \simeq y$.\\[1mm]
$(i)$ $(\eqv_{x})^{-1} \cong \eqv_{x}$.\\[1mm]
$(ii)$ $(e^{-1})^{-1} \cong e$.\\[1mm]
$(iii)$ If $e \cong d$, then $e^{-1} \cong d^{-1}$.
 
\end{proposition}

\begin{proof}(i) By $\Typ_2$ we have that $\eqv_x \ast {\eqv_x}^-1 \cong \eqv_x$ and by $\Typ_1$ we also have
$\eqv_x \ast {\eqv_x}^{-1} \cong {\eqv_x}^{-1}$, hence we get $(\eqv_{x})^{-1} \cong \eqv_{x}$.\\
(ii) Since $(e^{-1})^{-1} \ast e^{-1} \cong \eqv_x$, by $\Typ_4$ we get
$\big((e^{-1})^{-1} \ast e^{-1}\big) \ast e \cong \eqv_x \ast e$, and consequently 
$(e^{-1})^{-1} \ast \eqv_y \cong e$, hence $(e^{-1})^{-1} \cong e$.\\
(iii) By $\Typ_4$ we get $e^{-1} \ast e \cong e^{-1} \ast d$, hence $e^{-1} \ast d \cong
\eqv_{y}$, therefore $(e^{-1} \ast d) \ast d^{-1} \cong \eqv_{y} \ast d^{-1}$ i.e., 
$e^{-1} \ast \eqv_{x} \cong d^{-1}$, and $e^{-1} \cong d^{-1}$. 
\end{proof}

\begin{example}\label{exm: typ0}
Using basic properties of the equality $p =_{x =_{A} y} q$, of the
concatenation $p \ast q$ and of the inversion $p^{-1}$ 
of equality terms, it is easy to see that the structure
$$\Aii_{0} \equiv \big(A, =_{\Aii_0}, \refl_{\Aii_0}, \ast, ^{-1_{\Aii_0}}, \cong_{\Aii_{0}}\big)$$
is a typoid, where the equivalence
$\cong_{\Aii_{0}} : \prod_{x, y : A}\prod_{e, e{'} : x =_{A} y}\Uii$ is defined by 
$$\cong_{\Aii_{0}}(x, y, e, e{'}) \equiv (e =_{x =_{A} y}
e{'}),$$
for every $x, y : A$ and $e, e{'} : x =_{A} y$.
We call $\Aii_{0}$ the \textit{equality} typoid, and its typoid structure the \textit{equality} typoid 
structure on $A$.
\end{example}

\begin{example}\label{exm: typ1}If $A, B : \Uii$, it is easy to see that the structure
$$\Fun(A, B) \equiv \big(A \rightarrow B, \simeq_{\mathsmaller{\Fun(A, B)}}, \eqv_{\mathsmaller{\Fun(A, B)}}, \ast_{\mathsmaller{\Fun(A, B)}}, ^{-1_{\mathsmaller{\Fun(A, B)}}},
\cong_{\mathsmaller{\Fun(A, B)}}\big)$$
is a typoid, where if $f, g : A \to B$, $H, H{'} : f  \simeq_{\mathsmaller{\Fun(A, B)}} g$, and $G :  g  \simeq_{\mathsmaller{\Fun(A, B)}} h$, we define
$$f \simeq_{\mathsmaller{\Fun(A, B)}} g \equiv \prod_{x : A} f(x) =_{B} g(x),$$
$$H \ast_{\mathsmaller{\Fun(A, B)}} G \equiv \lambda(x : A).(H(x) \ast G(x)),$$
$$ H^{-1_{\mathsmaller{\Fun(A, B)}}} \equiv \lambda(x : A).(H(x))^{-1},$$
$$\eqv_{f} \equiv \lambda(x : A).\refl_{f(x)},$$
$$H \cong_{\mathsmaller{\Fun(A, B)}} H{'} \equiv \prod_{x : A} H(x) =_{(f(x) =_{B} g(x))} H{'}(x).$$
We call $\Fun(A, B)$ the \textit{typoid of functions} from $A$ to $B$.
Similarly, one can define a typoid structure on the dependent functions $\prod_{x : A}P(x)$, where 
$P: A \rightarrow \Uii$ is a type family over $A$.
\end{example}

\begin{example}\label{exm: typ2}Using Voevodsky's definition 
of the type $\isequiv(f)$ ``$f$ is an equivalence between $A, B : \Uii$'' (see~\cite{HoTT13}, section 2.4), 
it is easy to show that 
$$\Uni \equiv \big(\Uii, \simeq_{\mathsmaller{\Uni}}, \eqv_{\mathsmaller{\Uni}}, \ast_{\mathsmaller{\Uni}}, ^{-1_{\mathsmaller{\Uni}}}, \cong_{\mathsmaller{\Uni}}\big)$$
is a typoid, where if $(f, u), (f{'}, u{'}) : A  \simeq_{\mathsmaller{\Uni}} B$ and $(g, v) :  B  \simeq_{\mathsmaller{\Uni}} C$, we define
$$A \simeq_{\Uni} B \equiv \sum_{f: A \rightarrow B} \isequiv(f),$$
$$(f, u) \ast_{\mathsmaller{\Uni}} (g, v) \equiv (g \circ f, w),$$
$$(f, u)^{-1_{\mathsmaller{\Uni}}} \equiv (f^{-1}, u^{-1}),$$
$$\eqv_{A} \equiv (\id_{A}, i),$$
$$(f, u) \cong_{\mathsmaller{\Uni}} (f{'}, u{'}) \equiv \prod_{x : A} f(x) =_{B} f{'}(x),$$
where $w : \isequiv(g \circ f), u^{-1} : \isequiv(f^{-1})$ and $i : \isequiv(\id_{A})$.
We call $\Uni$ the \textit{universal typoid}.
\end{example}

\begin{definition}\label{def: ebf} 
If $\Aii, \Bii$ are typoids, we call a function $f: A \rightarrow B$ a \textit{typoid function}, if
there are dependent functions
$$\Phi_{f} : \prod_{x, y :A}\prod_{e: x \simeq_{\mathsmaller{\Aii}} y}f(x) \simeq_{\mathsmaller{\Bii}} f(y),$$
$$\Phi_{f}^{2} : \prod_{x, y :A}\prod_{e, d: x \simeq_{\mathsmaller{\Aii}} y}\prod_{i: e \cong_{\mathsmaller{\Aii}} d}\Phi_{f}(x, y, e) \cong_{\mathsmaller{\Bii}} \Phi_{f}(x, y, d),$$
which we call an \textit{$1$-associate} of $f$ and a \textit{$2$-associate} of $f$ with respect to 
$\Phi_{f}$, respectively, such that for every $x, y, z :A$ and every
$e_{1} : x \simeq_{\mathsmaller{\Aii}} y$, $e_{2} : y \simeq_{\mathsmaller{\Aii}} z$ the following conditions hold.\\[1mm]
$(i)$ $\Phi_{f}(x, x, \eqv_{x}) \cong_{\mathsmaller{\Bii}} \eqv_{f(x)},$\\[1mm]
$(ii)$ $\Phi_{f}(x, z, e_{1} \ast_{\mathsmaller{\Aii}} e_{2}) \cong_{\mathsmaller{\Bii}} \Phi_{f}(x, y, e_{1}) \ast_{\mathsmaller{\Bii}} \Phi_{f}(y, z, e_{2})$.\\[1mm]
If $\Phi_{f}(x, x, \eqv_{x}) \equiv \eqv_{f(x)},$ for every $x : A$, we call $f$ \textit{strict} 
with respect to $\Phi_{f}$.  
\end{definition}

$\Phi_{f}$ witnesses that $f$ preserves the equivalences between the terms of type $A$ and $B$,
as $\Phi_{f}(x, y) : x \simeq_{\mathsmaller{\Aii}} y \to f(x) \simeq_{\mathsmaller{\Bii}} f(y),$
while $\Phi_{f}^2$ witnesses that $\Phi_f$ preserves the equivalences between equivalences, as
$\Phi_{f}^2(x, y, e, d) : e \cong_{\mathsmaller{\Aii}} d \to \Phi_{f}(x, y, e) \cong_{\mathsmaller{\Bii}} \Phi_{f}(x, y, d).$

\begin{proposition}\label{prp: inv2} If $\Aii, \Bii$ are typoids and $f : A \rightarrow B$ is a typoid function,
then, for every $x, y : A$ and $e : x \simeq_{\mathsmaller{\Aii}} y$, we have that
$$\Phi_{f}(y, x, e^{-1_{\mathsmaller{\Aii}}}) \cong_{\mathsmaller{\Bii}} [\Phi_{f}(x, y, e)]^{-1_{\mathsmaller{\Bii}}}.$$

\end{proposition}

%
%

\begin{example}\label{exm: typf} If $\Aii_{0}, \Bii_{0}$ are equality typoids and $f : A \rightarrow B$, 
then $f$ is a strict typoid function
with respect to its $1$-associate, the application function $\ap_{f}$, and with 
$2$-associate with respect to $\ap_{f}$ the two-dimensional application function $\ap_{f}^{2}$
of $f$, where 
$$\ap_{f} : \prod_{x, y :A}\prod_{p: x =_{A} y}f(x) =_{B} f(y),$$
$$\ap_{f}^{2} : \prod_{x, y :A}\prod_{p, q: x =_{A} y}\prod_{r: p =_{\tiny{(x =_{A} y)}} q}\ap_{f}(x, y, p)
=_{\tiny{(f(x) =_{B} f(y))}} \ap_{f}(x, y, q).$$
The properties $\ap_{f}(x, x, \refl_{x}) \equiv \refl_{f(x)}$ and
$\ap_{f}(x, z, p \ast q) = \ap_{f}(x, y, p) \ast \ap_{f}(y, z, q)$ follow 
from the $J$-rule (see section 2.2 of~\cite{HoTT13}).
\end{example}

\begin{proposition}\label{prp: comp}If $\Aii, \Bii, \Cii$ are typoids and $f: A \rightarrow B, g: B \rightarrow C$
are typoid functions with associates $\Phi_{f}, \Phi_{f}^{2}$ and $\Phi_{g}, \Phi_{g}^{2}$, respectively, then
$g \circ f : A \rightarrow C$ is a typoid function with associates 
$$\Phi_{g \circ f} : \prod_{x, y :A}\prod_{e: x \simeq_{\mathsmaller{\Aii}} y}g(f(x)) \simeq_{\mathsmaller{\Cii}} g(f(y)),$$
$$\Phi_{g \circ f}^{2} : \prod_{x, y :A}\prod_{e, d: x \simeq_{\mathsmaller{\Aii}} y}\prod_{i: e \cong_{\mathsmaller{\Aii}} d}\Phi_{g \circ f}(x, y, e) \cong_{\mathsmaller{\Cii}} \Phi_{g \circ f}(x, y, d),$$
defined for every $x, y : A, e, d : x \simeq_{\mathsmaller{\Aii}} y, i: e \cong_{\mathsmaller{\Aii}} d$ by
$$\Phi_{g \circ f}(x, y, e) \equiv \Phi_{g}\big(f(x), f(y), \Phi_{f}(x, y, e)\big),$$
$$\Phi_{g \circ f}^{2}(x, y, e, d, i) \equiv \Phi_{g}^{2}\bigg(f(x), f(y), \Phi_{f}(x, y, e),  
\Phi_{f}(x, y, d), \Phi_{f}^{2}(x, y, e, d, i)\bigg).$$
If $f, g$ are strict with respect to $\Phi_{f}, \Phi_{g}$, then
$g \circ f$ is strict with respect to $\Phi_{g \circ f}$.
\end{proposition}

If $P : A \to \Uii$ and $p : \Exy$, the $J$-rule implies the existence of the transport function
$p_*^P : P(x) \to P(y)$ over $p$ (see~\cite{HoTT13}, section 2.3).

\begin{proposition}\label{prp: idtoeqv} If $\Aii$ is a
typoid, the identity $\id_{A}: A \rightarrow A$ is a strict
typoid function from $\Aii_{0}$ to $\Aii$ with respect to its $1$-associate 
$$\idtoeqv_{\mathsmaller{\Aii}} : \prod_{x, y :A}\prod_{p: x =_{A} y}x \simeq_{\mathsmaller{\Aii}} y,$$  
$$\idtoeqv_{\mathsmaller{\Aii}}(x, y, p) \equiv p_{\ast}^{P_{x}}(\eqv_{x}),$$ 
where $P_{x} : A \rightarrow \Uii$ is defined by $P_{x}(z) \equiv x \simeq_{\mathsmaller{\Aii}} z$, for every 
$z : A$.
\end{proposition}

If we consider $A$ with the equality typoid structure as a codomain of $\id_{A}$,
then by Lemma 2.11.2 in~\cite{HoTT13} we have that
$$\idtoeqv_{\mathsmaller{\Aii_{0}}}(x, y, p) \equiv p_{\ast}^{P_{x}}(\refl_{x}) \equiv p_{\ast}^{z \mapsto x =_{A} z}(\refl_{x}) = \refl_{x} \ast p = p$$
i.e., $\idtoeqv_{\mathsmaller{\Aii_{0}}}(x, y, p)$ is pointwise equal to $\id_{x =_{A} y}$. In~\cite{HoTT13}, 
section 2.10, the function $\idtoeqv : A =_{\Uii} B \rightarrow A \simeq_{\Uii} B$ is defined by 
$\idtoeqv(p) \equiv p_{\ast}^{\id_{\Uii}},$ 
for every $p : A =_{\Uii} B$. Since 
$$\idtoeqv(\refl_{A}) \equiv (\refl_{A})_{\ast}^{\id_{\Uii}} \equiv \id_{\id_{\Uii}(A)} \equiv \id_{A},$$
and
$$\idtoeqv_{\mathsmaller{\Uni}}(A, A, \refl_{A}) \equiv
(\refl_{A})_{\ast}^{P_{A}}(\id_{A}) \equiv \id_{P_{A}(A)}(\id_{A}) \equiv \id_{A},$$ 
the two functions agree on $\refl_{A}$,
hence by the $J$-rule
they are pointwise equal. The definition of $\idtoeqv_{\mathsmaller{\Aii}}$ is a common formulation of the functions $\happly$, related to the axiom of function extensionality, and $\idtoeqv$, related to the univalence axiom.

\section{Product typoid}
\label{sec:product}

\begin{proposition}\label{prp: proj1} If $(A, \simeq_{\mathsmaller{\Aii}}), (B, \simeq_{\mathsmaller{\Bii}})$ are setoids, there are dependent functions
$$T : \prod_{z, w : A \times B}\bigg((\pr_{1}(z) \simeq_{\mathsmaller{\Aii}} \pr_{1}(w))
\times (\pr_{2}(z) \simeq_{\mathsmaller{\Bii}} \pr_{2}(w)) \rightarrow z \simeq_{\mathsmaller{\Aii \times \Bii}} w\bigg),$$
$$\Upsilon : \prod_{z, w : A \times B}\bigg(z \simeq_{\mathsmaller{\Aii \times \Bii}} w \rightarrow (\pr_{1}(z) \simeq_{\mathsmaller{\Aii}} \pr_{1}(w))
\times (\pr_{2}(z) \simeq_{\mathsmaller{\Bii}} \pr_{2}(w))\bigg),$$
where 
$$(x, y) \simeq_{\mathsmaller{\Aii \times \Bii}} (x{'}, y{'}) \equiv (x \simeq_{A} x{'}) \times (y \simeq_{B} y{'}),$$ 
such that for each $i \in \{1, 2\}$, where $C_{1} \equiv A$ and $C_{2} \equiv B$, 
$$\prod_{z, w : A \times B}\pr_{i}\bigg(\Upsilon(z, w, T(z, w, e_{1}, e_{2}))\bigg) \cong_{\mathsmaller{C_{i}}} e_{i}.$$  
 
\end{proposition}

\begin{proof} 
If $x: A$ and $y : B$, we find $G(x, y) : \prod_{w : A \times B}P(w)$, where 
$$P(w) \equiv(x, y) \simeq_{\mathsmaller{\Aii \times \Bii}} w \rightarrow (x \simeq_{\mathsmaller{\Aii}} \pr_{1}(w))
\times (y \simeq_{\mathsmaller{\Bii}} \pr_{2}(w)),$$
hence by the induction principle of the product type we get $\Upsilon$. We define the dependent function
$H : \prod_{x{'} : A}\prod_{y{'} : B}P((x{'}, y{'}))$ by 
$H(x{'}, y{'}) \equiv \id_{(x, y) \simeq_{\mathsmaller{\Aii \times \Bii}} (x{'}, y{'})}$, and then
$H(x{'}, y{'}) : (x, y) \simeq_{\mathsmaller{\Aii \times \Bii}} (x{'}, y{'}) \leftrightarrow (x \simeq_{\mathsmaller{\Aii}} x{'}) \times (y \simeq_{\mathsmaller{\Aii}} y{'})$, and
$G(x, y)$ is given again by the induction principle of the product type. For $T$ 
we proceed similarly.
\end{proof}

\begin{corollary}\label{crl: proj2} If  $(A, \simeq_{\mathsmaller{\Aii}}), (B, \simeq_{\mathsmaller{\Bii}})$ are setoids, then $\pr_{1}, \pr_{2}$ are setoid functions.
 
\end{corollary}

\begin{proof} Let the dependent function $\Phi_{\pr_{1}} : \prod_{z, w : A \times B}\prod_{e: z \simeq_{\mathsmaller{A \times B}} w}\pr_{1}(z)  \simeq_{\mathsmaller{\Aii}} \pr_{1}(w)$ by $\Phi_{\pr_{1}}(z, w, e) \equiv \pr_{1}(\Upsilon(z, w, e))$. Similarly, we define $\Phi_{\pr_{2}}(z, w, e) \equiv \pr_{2}(\Upsilon(z, w, e))$.
\end{proof}

We use the notations $e_{1} \equiv \Phi_{\pr_{1}}(z, w, e)$ and $e_{2} \equiv \Phi_{\pr_{2}}(z, w, e)$.

\begin{proposition}\label{prp: prodebt} If $\Aii, \Bii$ are typoids, then the structure
$$\Aii \times \Bii \equiv \big(A \times B, \simeq_{\mathsmaller{\Aii \times \Bii}}, \eqv_{\mathsmaller{\Aii \times \Bii}}, \ast_{\mathsmaller{\Aii \times \Bii}}, ^{-1_{\mathsmaller{\Aii \times \Bii}}}, 
\cong_{\mathsmaller{\Aii \times \Bii}}\big)$$
is a typoid, where for every $z, w, u : A \times B$ and $e, e{'} : z =_{A \times B} w, 
d : w =_{A \times B} u$ we define
$$\eqv_{z} \equiv T(z, z, \eqv_{\pr_{1}(z)}, \eqv_{\pr_{2}(z)}),$$
$$e \ast_{\mathsmaller{\Aii \times \Bii}} d \equiv T(z, u, e_{1} \ast_{\mathsmaller{\Aii}} d_{1}, e_{2} \ast_{\mathsmaller{\Bii}} d_{2}),$$
$$e^{-1_{\mathsmaller{\Aii \times \Bii}}} \equiv T(w, z, e_{1}^{-1_{\mathsmaller{\Aii}}}, e_{2}^{-1_{\mathsmaller{\Bii}}}),$$
$$e \cong_{\mathsmaller{\Aii \times \Bii}} e{'} \equiv (e_{1} \cong_{\mathsmaller{\Aii}} e_{1}{'}) \times (e_{2} \cong_{\mathsmaller{\Bii}} e_{2}{'}).$$

\end{proposition}



\begin{corollary}\label{crl: proj3} If  $\Aii, \Bii$ are typoids, then
$\pr_{1}, \pr_{2}$ are typoid functions.
 
\end{corollary}

\begin{proof} We show this only for $\pr_{1}$. The equivalence $\Phi_{\pr_{1}}(z, z, \eqv_{z}) \cong_{\mathsmaller{\Aii}} \eqv_{\pr_{1}(z)}$
follows from the previously shown equivalence $(\eqv_{z})_{1} \cong_{\mathsmaller{\Aii}} \eqv_{\pr_{1}(z)}$. The equivalence
$\Phi_{\pr_{1}}(z, u, e \ast_{\mathsmaller{\Aii \times \Bii}} d) \cong_{\mathsmaller{\Aii}} \Phi_{\pr_{1}}(z, w, e) \ast_{\mathsmaller{\Aii}} \Phi_{\pr_{1}}(w, u, d)$ 
is rewritten as $(e \ast_{\mathsmaller{\Aii \times \Bii}} d)_{1} \cong_{\mathsmaller{\Aii}} (e_{1} \ast_{\mathsmaller{\Aii}} d_{1})$, which follows immediately
by the definition of $e \ast_{\mathsmaller{\Aii \times \Bii}} d$ and the relation between $\Upsilon$ and $T$ in Proposition~\ref{prp: proj1}.
Since $e_{1} \equiv \Phi_{\pr_{1}}(z, w, e)$ and $d_{1} \equiv \Phi_{\pr_{1}}(z, w, d)$, by the definition of 
$e \cong_{\mathsmaller{\Aii \times \Bii}} d$ we get the following
$2$-associate of $\pr_{1}$ with respect to $\Phi_{\pr_{1}}$
$$\Phi_{\pr_{1}}^{2} : \prod_{z, w : A \times B}\prod_{e, d: z \simeq_{\mathsmaller{\Aii \times \Bii}} w}\prod_{i: e 
\cong_{\mathsmaller{\Aii \times \Bii}} d}\Phi_{\pr_{1}}(z, w, e) \cong_{\mathsmaller{\Aii}} \Phi_{\pr_{1}}(z, w, d),$$
$$\Phi_{\pr_{1}}^{2}(z, w, e, d, i) \equiv \pr_{1}(i).$$
\end{proof}

\begin{corollary}\label{crl: T2}If  $\Aii, \Bii$ are typoids, $z, w : A \times B$,  and  $e : z \simeq_{\mathsmaller{\Aii \times \Bii}} w$, then 
$$T(z, w, e_{1}, e_{2}) \cong_{\mathsmaller{\Aii \times \Bii}} e.$$ 
 
\end{corollary}

\begin{proof}If $\tau \equiv T(z, w, e_{1}, e_{2})$, then $\tau \cong_{\mathsmaller{\Aii \times \Bii}} e \equiv (\tau_{1} 
\cong_{\mathsmaller{\Aii}} e_{1}) \times (\tau_{2} \cong_{\mathsmaller{\Aii}} e_{2})$. By Proposition~\ref{prp: proj1} $\tau_{1} \equiv \pr_{1}(\Upsilon(z, w, T(z, w, e_{1}, e_{2}))) \cong_{\mathsmaller{\Aii}} e_{1}$ and similarly $\tau_{2} \cong_{\mathsmaller{\Bii}} e_{2}$.
\end{proof}

\begin{corollary}\label{crl: T}If  $\Aii, \Bii$ are typoids, $z, w : A \times B$,  and  $e_{1}, d_{1} : \pr_{1}(z) \simeq_{\mathsmaller{\Aii}} \pr_{1}(w)$,  $e_{2}, d_{2} : \pr_{2}(z) \simeq_{\mathsmaller{\Aii}} \pr_{2}(w)$ such that $e_{1} \cong_{\mathsmaller{\Aii}} d_{1}$ and 
$e_{2} \cong_{\mathsmaller{\Aii}} d_{2}$, then
$$T(z, w, e_{1}, e_{2}) \cong_{\mathsmaller{\Aii \times \Bii}} T(z, w, d_{1}, d_{2}).$$ 
 
\end{corollary}

\begin{proof}
If $\tau \equiv T(z, w, e_{1}, e_{2})$ and $\tau{'} \equiv T(z, w, d_{1}, d_{2})$, then 
by Proposition~\ref{prp: proj1} we get 
$$\tau_{1} \equiv \pr_{1}(\Upsilon(z, w, T(z, w, e_{1}, e_{2})) \cong_{\mathsmaller{\Aii}} e_{1} \ \& \ 
\tau_{2} \equiv \pr_{2}(\Upsilon(z, w, T(z, w, e_{1}, e_{2})) \cong_{\mathsmaller{\Bii}} e_{2}.$$
Similarly we get $\tau_{1}{'} \equiv \pr_{1}(\Upsilon(z, w, T(z, w, d_{1}, d_{2})) \cong_{\mathsmaller{\Aii}} d_{1}$
and $\tau_{2}{'} \equiv \pr_{2}(\Upsilon(z, w, T(z, w, d_{1}, d_{2})) \cong_{\mathsmaller{\Bii}} d_{2}$.
By our hypothesis and the definition of $\cong_{\mathsmaller{\Aii \times \Bii}}$ we get $\tau \cong_{\mathsmaller{\Aii \times \Bii}} \tau{'}$. 
\end{proof}

\section{Univalent typoids}
\label{sec: ut}

\begin{definition}\label{def: ut} A typoid $\Aii$ is called \textit{univalent}, if there are dependent functions
$$\Ua_{\mathsmaller{\Aii}} : \prod_{x, y : A}\prod_{e: x \simeq_{\mathsmaller{\Aii}} y}x = _{A} y,$$
$$\Ua_{\mathsmaller{\Aii}}^{2} : \prod_{x, y : A}\prod_{e, d: x \simeq_{\mathsmaller{\Aii}} y}\prod_{i : e \cong_{\mathsmaller{\Aii}} d}\Ua_{\mathsmaller{\Aii}}(x, y, e) = \Ua_{\mathsmaller{\Aii}}(x, y, d)$$
such that for every $x, y : A, p: x =_{A} y$ and $e: x \simeq_{\mathsmaller{\Aii}} y$ we have that 
$$\Ua_{\mathsmaller{\Aii}}(x, y, \Idtoeqv_{\mathsmaller{\Aii}}(x, y, p)) = p,$$
$$\Idtoeqv_{\mathsmaller{\Aii}}(x, y, \Ua_{\mathsmaller{\Aii}}(x, y, e)) \cong_{\mathsmaller{\Aii}} e,$$
where $\Idtoeqv_{\mathsmaller{\Aii}}$ is an $1$-associate of $\id_{A}$ (from $\Aii_{0}$ to $\Aii$) with respect to which $\id_{A}$ is 
strict\footnote{In Proposition~\ref{prp: idtoeqv} we showed the existence of such an associate of $\id_{A}$ but in general
we don't need to use the specific definition of $\idtoeqv_{\mathsmaller{\Aii}}$. This is crucial in proving that the product of univalent typoids is a univalent typoid. For this reason we use a different notation for this abstract $1$-associate of $\id_{A}$. Using the $J$-rule one shows that $\idtoeqv_{\Aii}$ and $\Idtoeqv_{\Aii}$ are pointwise equal, hence by function extensionality they are equal.}.
We call a univalent typoid \textit{strictly} univalent, if $\Ua_{\mathsmaller{\Aii}}(x, x, \eqv_{x}) \equiv \refl_{x}$. 
 
\end{definition}

The equality typoid $\Aii_{0}$ is strictly univalent, if we consider 
$$\Idtoeqv_{\Aii_0}(x, y, p) \equiv p
\equiv \Ua_{\Aii_0}(x, y, p),$$
for every $x, y : A$ and $p: x \simeq_{\Aii_0} y$.
The function extensionality axiom guarantees that the typoid of functions $\Fun(A, B)$ is
univalent, and Voevodsky's univalence axiom that the universal typoid $\Uni$ is univalent.
We need only to explain why the functions
$\funext$ and $\ua$ satisfy the above conditions: if $H, H{'} : f \simeq_{\mathsmaller{\Fun(A, B)}} g$ such that
$H \cong_{\mathsmaller{\Fun(A, B)}} H{'}$, then $\funext(H) = \funext(H{'})$,
and if $(f, u), (g, w) : A \simeq_{\mathsmaller{\Uni}} B$ such that $(f, u) \cong_{\mathsmaller{\Uni}} (g, w)$, then $\ua((f, u)) = \ua((g, w))$, respectively. By the 
function extensionality axiom if $H, H{'} : f \simeq_{\mathsmaller{\Fun(A, B)}} g$, then there is $p: H = H{'}$,
hence the application function of $\funext$ satisfies $\ap_{\funext}(p) : \funext(H) = \funext(H{'})$.
By Theorem 2.7.2 of~\cite{HoTT13} we
have that
$$\big((f, u) =_{A \simeq_{\mathsmaller{\Uni}} B} (g, w)\big) \simeq_{\mathsmaller{\Uni}} \sum_{p: f = g}\bigg(p_{\ast}^{f \mapsto \isequiv(f)}(u) = w\bigg).$$
By the function extensionality axiom the hypothesis $(f, u) \cong_{\mathsmaller{\Uni}} (g, w)$ implies that $f = g$, 
while a term of type $p_{\ast}^{f \mapsto \isequiv(f)}(u) = w$ is found by the equality of all terms of type $\isequiv(g)$. Hence the hypothesis $(f, u) \cong_{\Uni} (g, w)$ implies $(f, u) =_{A \simeq_{U} B} (g, w)$ and we use the application function of $\ua$ to get a term of type $\ua((f, u)) = \ua((g, w))$.

The next proposition is a common reformulation of properties of the functions $\funext$ and $\ua$ found  in
sections 2.9 and 2.10 of~\cite{HoTT13}, respectively.

\begin{proposition}\label{prp: ut1} If $\Aii$ is a univalent 
typoid, the identity function $\id_{A}: A \rightarrow A$ is a
typoid function from $\Aii$ to $\Aii_{0}$, with $\Ua_{\mathsmaller{\Aii}}^{2}$ as a $2$-associate of $\id_{A}$
with respect to its $1$-associate $\Ua_{\mathsmaller{\Aii}}$.

\end{proposition}

\begin{proof} We show that $\Phi_{\id_{A}} \equiv \Ua_{\mathsmaller{\Aii}}$ and $\Phi_{\id_{A}}^{2} \equiv \Ua_{\mathsmaller{\Aii}}^{2}$ are 
associates of $\id_{A}$. By definition we have that
$\Ua_{\mathsmaller{\Aii}}(x, x, \eqv_{x}) \equiv \Ua_{\mathsmaller{\Aii}}(x, x, \Idtoeqv_{\mathsmaller{\Aii}}(x, x, \refl_{x})) = \refl_{x}$. 
Since $\Idtoeqv_{\mathsmaller{\Aii}}(x, y, \Ua_{\mathsmaller{\Aii}}(x, y, e_{1})) \cong_{\mathsmaller{\Aii}} e_{1}$ and $\Idtoeqv_{\mathsmaller{\Aii}}(y, z, \Ua_{\mathsmaller{\Aii}}(y, z, e_{2})) 
\cong_{\mathsmaller{\Aii}} e_{2}$, by $\Typ_4$ we get
\begin{align*}
 e_{1} \ast_{\mathsmaller{\Aii}} e_{2}  & \equiv_{\mathsmaller{\Aii}}
 \Idtoeqv_{\mathsmaller{\Aii}}(x, y, \Ua_{\mathsmaller{\Aii}}(x, y, e_{1})) \ast_{\mathsmaller{\Aii}} \Idtoeqv_{\mathsmaller{\Aii}}(y, z, \Ua_{\mathsmaller{\Aii}}(y, z, e_{2}))\\
 & \equiv_{\mathsmaller{\Aii}} \Idtoeqv_{\mathsmaller{\Aii}}(x, z, [\Ua_{\mathsmaller{\Aii}}(x, y, e_{1}) \ast_{\mathsmaller{\Aii}} \Ua_{\mathsmaller{\Aii}}(y, z, e_{2})]))
\end{align*}
If $B \equiv \Ua_{\mathsmaller{\Aii}}\big(x, z, \Idtoeqv_{\mathsmaller{\Aii}}(x, z, [\Ua_{\mathsmaller{\Aii}}(x, y, e_{1}) \ast_{\mathsmaller{\Aii}} \Ua_{\mathsmaller{\Aii}}(y, z, e_{2})]))\big)$,
by the existence of $\Ua_{\mathsmaller{\Aii}}^{2}$ we get a term of type
$\Ua_{\mathsmaller{\Aii}}(x, z, e_{1} \ast_{\mathsmaller{\Aii}} e_{2})  = B$, hence a term of type 
$\Ua_{\mathsmaller{\Aii}}(x, z, e_{1} \ast_{\mathsmaller{\Aii}} e_{2}) = \Ua_{\mathsmaller{\Aii}}(x, y, e_{1}) \ast_{\mathsmaller{\Aii}} \Ua_{\mathsmaller{\Aii}}(y, z, e_{2}).$
\end{proof}

\begin{theorem}\label{thm: ut2}Let $\Aii, \Bii$ be typoids and $f : A \rightarrow B$.\\[1mm]
$(i)$ If $\Aii$ is univalent, then $f$ is a typoid function.\\
$(ii)$ If $\Aii$ is strictly univalent, then $f$ is a strict typoid function with respect to its
$1$-associate given in the proof of $(i)$.
 
\end{theorem}

\begin{proof}(i) Through the correspondences
$$x \simeq_{\mathsmaller{\Aii}} y \stackrel{\Ua_{\mathsmaller{\Aii}}(x, y)} \longrightarrow x =_{A} y \stackrel{\ap_{f}(x, y)} \longrightarrow f(x) =_{B} f(y)
\stackrel{\Idtoeqv_{\mathsmaller{\Bii}}(f(x), f(y))} \longrightarrow f(x) \simeq_{\mathsmaller{\Bii}} f(y)$$
we define the dependent function $\Phi_{f} : \prod_{x, y :A}\prod_{e : x \simeq_{\mathsmaller{\Aii}} y}f(x) \simeq_{\mathsmaller{\Bii}} f(y)$ by
$$\Phi_{f}(x, y, e) \equiv \Idtoeqv_{\mathsmaller{\Bii}}\Big(f(x), f(y), \ap_{f}(x, y, \Ua_{\mathsmaller{\Aii}}(x, y, e))\Big).$$
By the proof of Proposition~\ref{prp: ut1} there is $r : \Ua_{\mathsmaller{\Aii}}(x, x, \eqv_{x}) = \refl_{x}$, and
by $\ap_{f}^{2}$ we get a term $r{'} : \ap_{f}(x, x, \Ua_{\mathsmaller{\Aii}}(x, x, \eqv_{x})) = \ap_{f}(x, x, \refl_{x}) \equiv \refl_{f(x)}$. Since $\Idtoeqv_{\mathsmaller{\Bii}}^{2}$ is of type
$$\prod_{x{'}, y{'} : B}\prod_{p{'}, q{'}: x{'} =_{B} y{'}}\prod_{r{'} : p{'} = 
q{'}}\Idtoeqv_{\mathsmaller{\Bii}}(x{'}, y{'}, p{'}) \cong_{\mathsmaller{\Bii}} \Idtoeqv_{\mathsmaller{\Bii}}(x{'}, y{'}, q{'}),$$
$\Idtoeqv_{\mathsmaller{\Bii}}^{2}(f(x), f(x), \ap_{f}(x, x, \Ua_{\mathsmaller{\Aii}}(x, x, \eqv_{x})), \refl_{f(x)}, r{'})$
is of type
$$\Idtoeqv_{\mathsmaller{\Bii}}(f(x), f(x), \ap_{f}(x, x, \Ua_{\mathsmaller{\Aii}}(x, x, \eqv_{x}))) \cong_{\mathsmaller{\Bii}} \Idtoeqv_{\mathsmaller{\Bii}}(f(x), f(x), \refl_{f(x)}).$$
By the definition of $\Phi_{f}$ and as $\Idtoeqv_{\mathsmaller{\Bii}}(f(x), f(x), \refl_{f(x)}) \equiv \eqv_{f(x)}$ we get
$\Phi_{f}(x, x, \eqv_{x}) \cong_{\mathsmaller{\Bii}} \eqv_{f(x)}$. By the existence of $\ap_{f}^{2}$ and $\Idtoeqv_{\mathsmaller{\Bii}}^{2}$ 
we get
\begin{align*}
 \Phi_{f}(x, z, e \ast_{\mathsmaller{\Aii}} d) & \equiv \Idtoeqv_{\mathsmaller{\Bii}}\Big(f(x), f(z), \ap_{f}(x, z, \Ua_{\mathsmaller{\Aii}}(x, z, e \ast_{\mathsmaller{\Aii}} d))\Big)\\
 & \cong_{\mathsmaller{\Bii}} \Idtoeqv_{\mathsmaller{\Bii}}\Big(f(x), f(z), \ap_{f}(x, z, \Ua_{\mathsmaller{\Aii}}(x, y, e) \ast_{\mathsmaller{\Bii}} \Ua_{\mathsmaller{\Aii}}(y, z,d))\Big)\\
 & \cong_{\mathsmaller{\Bii}} \Idtoeqv_{\mathsmaller{\Bii}}\Big(f(x), f(z), \ap_{f}(x, y, \Ua_{\mathsmaller{\Aii}}(x, y, e)) \ast_{\mathsmaller{\Bii}}  \ap_{f}(y, z,\Ua_{\mathsmaller{\Aii}}(y, z,d))\Big)\\
 & \cong_{\mathsmaller{\Bii}} \Idtoeqv_{\mathsmaller{\Bii}}\Big(f(x), f(y), \ap_{f}(x, y, \Ua_{\mathsmaller{\Aii}}(x, y, e))\Big) \ast_{\mathsmaller{\Bii}}\\
 & \ \ \ \ \ \ \Idtoeqv_{\mathsmaller{\Bii}}\Big(f(y), f(z), \ap_{f}(y, z, \Ua_{\mathsmaller{\Aii}}(y, z, d))\Big)\\
 & \equiv \Phi_{f}(x, y, e) \ast_{\mathsmaller{\Bii}} \Phi_{f}(y, z, d).
\end{align*}
We define $\Phi_{f}^{2} : \prod_{x, y :A}\prod_{e, d: x \simeq_{\mathsmaller{\Aii}} y}\prod_{i: e \cong_{\mathsmaller{\Aii}} d}\Phi_{f}(x, y, e) \cong_{\mathsmaller{\Bii}} \Phi_{f}(x, y, d)$ by
\begin{align*}
 & \Phi_{f}^{2}(x, y, e, d, i)  \equiv \Idtoeqv_{\mathsmaller{\Bii}}^{2}\bigg(f(x), f(y), \ap_{f}(x, y, \Ua_{\mathsmaller{\Aii}}(x, y, e)), \ap_{f}(x, y, \Ua_{\mathsmaller{\Aii}}(x, y, d)),\\
 & \ap_{f}^{2}\Big(x, y, \ap_{f}(x, y, \Ua_{\mathsmaller{\Aii}}(x, y, e)), \ap_{f}(x, y, \Ua_{\mathsmaller{\Aii}}(x, y, d)), 
 \Ua_{\mathsmaller{\Aii}}^{2}(x, y, e, d, i)\Big)\bigg). 
\end{align*}
Since the term $\Ua_{\mathsmaller{\Aii}}^{2}(x, y, e, d, i)$ is of type $\Ua_{\mathsmaller{\Aii}}(x, y, e) = \Ua_{\mathsmaller{\Aii}}(x, y, d)$ and the terms
$\ap_{f}(x, y, \Ua_{\mathsmaller{\Aii}}(x, y, e))$ and $\ap_{f}(x, y, \Ua_{\mathsmaller{\Aii}}(x, y, d))$ are of type $f(x) = _{B} f(y)$, the term
$$\ap_{f}^{2}\Big(x, y, \ap_{f}(x, y, \Ua_{\mathsmaller{\Aii}}(x, y, e)), \ap_{f}(x, y, \Ua_{\mathsmaller{\Aii}}(x, y, d)), 
 \Ua_{\mathsmaller{\Aii}}^{2}(x, y, e, d, i)\Big)$$
is of type $\ap_{f}(x, y, \Ua_{\mathsmaller{\Aii}}(x, y, e)) = \ap_{f}(x, y, \Ua_{\mathsmaller{\Aii}}(x, y, d))$. Hence by the type of $\Idtoeqv_{\mathsmaller{\Bii}}^{2}$ and the definition of $\Phi_{f}$ we get that $\Phi_{f}^{2}(x, y, e, d, i)$ 
is of type $\Phi_{f}(x, y, e) \cong_{\mathsmaller{\Bii}} \Phi_{f}(x, y, d)$.\\
(ii) By the proof of Proposition~\ref{prp: idtoeqv} we have that
\begin{align*}
 \Phi_{f}(x, x, \eqv_{x}) & \equiv \Idtoeqv_{\mathsmaller{\Bii}}\Big(f(x), f(x), \ap_{f}(x, x, \Ua_{\mathsmaller{\Aii}}(x, x, \eqv_{x}))\Big)\\
  & \equiv \Idtoeqv_{\mathsmaller{\Bii}}\Big(f(x), f(x), \ap_{f}(x, x, \refl_{x}))\Big)\\
 & \equiv \Idtoeqv_{\mathsmaller{\Bii}}(f(x), f(x), \refl_{f(x)}))\\
 & \equiv \eqv_{f(x)}.
 \end{align*}
\end{proof}


\begin{proposition}\label{prp: ut3} Let $\Aii, \Bii$ be typoids, $x, y: A$, and
$f: A \rightarrow B$ a typoid function with $1$-associate $\Phi_{f}$. If $\Bii$ is univalent, the following diagram commutes

\begin{center}
\begin{tikzcd}
x =_{A} y  \dar[swap]{\Idtoeqv_{A}(x, y)} \arrow{r}{\ap_{f}(x, y)} 
& f(x) =_{B} f(y)  \\
x \simeq_{\mathsmaller{\Aii}} y \arrow{r}[swap]{\Phi_{f}(x, y)}
& f(x) \simeq_{\mathsmaller{\Bii}} f(y). \arrow{u}[swap]{\Ua_{\mathsmaller{\Bii}}(f(x), f(y))}
\end{tikzcd}
\end{center}

\end{proposition}

\begin{proof}In order to use the $J$-rule, we define $C : \prod_{x, y : A}\prod_{p: x =_{A} y}\Uii$ by
$$C(x, y, p) \equiv \Ua_{\mathsmaller{\Bii}}\bigg(f(x), f(y), \Phi_{f}\big(x, y, 
\Idtoeqv_{\mathsmaller{\Aii}}(x, y, p)\big)\bigg) = \ap_{f}(x, y, p).$$
If $i : \Phi_{f}(x, x, \eqv_{x}) \cong_{\mathsmaller{\Bii}} \eqv_{f(x)}$, then 
$\Ua_{\mathsmaller{\Bii}}^{2}(f(x), f(x), \Phi_{f}(x, x, \eqv_{x}), \eqv_{f(x)}, i)$ is
a term of type
$$\Ua_{\mathsmaller{\Bii}}(f(x), f(x), \Phi_{f}(x, x, \eqv_{x})) = \Ua_{\mathsmaller{\Bii}}(f(x),
f(x), \eqv_{f(x)}).$$ Since we have that
$\Ua_{\mathsmaller{\Bii}}(f(x), f(x), \eqv_{f(x)}) = \refl_{f(x)}$, the type
\begin{align*}
C(x, x, \refl_{x}) & \equiv \Ua_{\mathsmaller{\Bii}}(f(x), f(x), \Phi_{f}(x, x, \Idtoeqv_{\mathsmaller{\Bii}}(x, 
x, \refl_{x}))) 
 = \ap_{f}(x, x, \refl_{x})\\
 & \equiv \Ua_{\mathsmaller{\Bii}}(f(x), f(x), \Phi_{f}(x, x, \eqv_{x})) = \ap_{f}(x, x, \refl_{x})
 \end{align*}
is inhabited, and the $J$-rule can be used.
\end{proof}

\begin{corollary}\label{crl: ut4} If $\Aii, \Bii$ are univalent typoids, $x, y: A$, and
$f: A \rightarrow B$, the following diagram commutes

\begin{center}
\begin{tikzcd}
x \simeq_{\mathsmaller{\Aii}} y  \dar[swap]{\Ua_(x, y)} \arrow{r}{\Phi_{f}(x, y)} 
& f(x) \simeq_{\mathsmaller{\Bii}} f(y) \arrow{d}{\Ua_{\mathsmaller{\Bii}}(f(x), f(y))} \\
x =_{A} y \arrow{r}[swap]{\ap_{f}(x, y)}
& f(x) =_{B} f(y).
\end{tikzcd}
\end{center}
where $\Phi_{f}$ is the $1$-associate of $f$ determined by Theorem~\ref{thm: ut2}. 
 
\end{corollary}

%

For the next two results we use the following dependent functions (see~\cite{HoTT13}, section 2.6)
$$\epair : \prod_{z, w : A \times B}\bigg(z =_{A \times B} w \rightarrow (\pr_{1}(z) =_{A} \pr_{1}(w))
\times (\pr_{2}(z) =_{B} \pr_{2}(w))\bigg),$$
$$\paire : \prod_{z, w : A \times B}\bigg((\pr_{1}(z) =_{A} \pr_{1}(w))
\times (\pr_{2}(z) =_{B} \pr_{2}(w)) \rightarrow z =_{A \times B} w\bigg),$$
which are the $=$-version of the functions of $\Upsilon$ and $T$, respectively, and they satisfy
$$\epair(z, z, \refl_{z}) \equiv (\refl_{\pr_{1}(z)}, \refl_{\pr_{2}(z)}),$$
$$\paire(z, z, (\refl_{\pr_{1}(z)}, \refl_{\pr_{2}(z)})) \equiv \refl_{z}.$$


\begin{lemma}\label{lem: idtoeqvprod}
 If $\Aii, \Bii$ are typoids such that $\id_{A}, \id_{B}$ are strict with respect to the their $1$-associates
 $\Idtoeqv_{\mathsmaller{\Aii}}, \Idtoeqv_{\mathsmaller{\Bii}}$, respectively, then $\id_{A \times B}$ is strict with respect to its $1$-associate
 $$\Idtoeqv_{\mathsmaller{\Aii \times \Bii}} : \prod_{z, w : A \times B}\prod_{p: z = _{A \times B} w}z \simeq_{\mathsmaller{\Aii \times \Bii}} w,$$
 $$\Idtoeqv_{\mathsmaller{\Aii \times \Bii}}(z, w, p) \equiv T(z, w, e_{1}, e_{2}),$$
 $$e_{1} \equiv \Idtoeqv_{\mathsmaller{\Aii}}(\pr_{1}(z), \pr_{1}(w), p_{1}), \ \ \ e_{2} \equiv \Idtoeqv_{\mathsmaller{\Bii}}(\pr_{2}(z), \pr_{2}(w), p_{2}),$$
 $$p_{1} \equiv \pr_{1}(\epair(z, w, p)), \ \ \  p_{2} \equiv \pr_{2}(\epair(z, w, p)).$$
\end{lemma}

\begin{theorem}\label{thm: utprod}
 If $\Aii, \Bii$ are univalent typoids, then $\Aii \times \Bii$ is a univalent typoid.
\end{theorem}

\begin{proof}We need to define dependent functions
$$\Ua_{\mathsmaller{\Aii \times \Bii}} : \prod_{z, w : A \times B}\prod_{e: z \simeq_{\mathsmaller{\Aii \times \Bii}} w}z = _{A \times B} w,$$
$$\Ua_{\mathsmaller{\Aii \times \Bii}}^{2} : \prod_{z, w : A \times B}\prod_{e, d: z \simeq_{\mathsmaller{\Aii \times \Bii}} e}\prod_{i : e \cong_{\mathsmaller{\Aii \times \Bii}} d}\Ua_{\mathsmaller{\Aii \times \Bii}}(z, w, e) = \Ua_{\mathsmaller{\Aii \times \Bii}}(z, w, d)$$
such that for every $z, w : A \times B, p: z =_{A \times B} w$ and $e: z \simeq_{\mathsmaller{\Aii \times \Bii}} w$ 
$$\Ua_{\mathsmaller{\Aii \times \Bii}}(z, w, \Idtoeqv_{\mathsmaller{\Aii \times \Bii}}(z, w, p)) = p \ \ \& \ \ 
\Idtoeqv_{\mathsmaller{\Aii \times \Bii}}(z, w, \Ua_{\mathsmaller{\Aii \times \Bii}}(z, w, e)) \cong_{\mathsmaller{\Aii \times \Bii}} e,$$
where $\Idtoeqv_{\mathsmaller{\Aii \times \Bii}}$ is defined in Lemma~\ref{lem: idtoeqvprod}.
If $e{'} : z \simeq_{\mathsmaller{\Aii \times \Bii}} w$, we define
$$\Ua_{\mathsmaller{\Aii \times \Bii}}(z, w, e{'}) \equiv \paire(z, w, p_{1}{'}, p_{2}{'}),$$
$$p_{1}{'} \equiv \Ua_{\mathsmaller{\Aii}}(\pr_{1}(z), \pr_{1}(w), e_{1}{'}) : \pr_{1}(z) =_{A} \pr_{1}(w),$$
$$p_{2}{'} \equiv \Ua_{\mathsmaller{\Bii}}(\pr_{2}(z), \pr_{2}(w), e_{2}{'}) : \pr_{2}(z) =_{B} \pr_{2}(w),$$
$$e_{1}{'} \equiv \pr_{1}(\Upsilon(z, w, e)) : \pr_{1}(z) \simeq_{\mathsmaller{\Aii}} \pr_{1}(w),$$
$$e_{2}{'} \equiv \pr_{2}(\Upsilon(z, w, e)) : \pr_{2}(z) \simeq_{\mathsmaller{\Aii}} \pr_{2}(w).$$
If $e{'} \equiv \Idtoeqv_{\mathsmaller{\Aii \times \Bii}}(z, w, p) \equiv T(z, w, e_{1}, e_{2})$, where $p : z =_{\Aii \times \Bii} w$ and
$e_{1}, e_{2}$ are defined in Lemma~\ref{lem: idtoeqvprod}, then by Proposition~\ref{prp: proj1}
$e_{1}{'} \equiv \pr_{1}(\Upsilon(z, w, T(z, w, e_{1}, e_{2}))) \cong_{\mathsmaller{\Aii}} e_{1}$,
and proceeding similarly we get $e_{2}{'} \cong_{\mathsmaller{\Bii}} e_{2}$. With the use of $\Ua_{\mathsmaller{\Aii}}^{2}$ we have that
\begin{align*}
 p_{1}{'} & \equiv \Ua_{\mathsmaller{\Aii}}(\pr_{1}(z), \pr_{1}(w), e_{1}{'})\\
 & = \Ua_{\mathsmaller{\Aii}}(\pr_{1}(z), \pr_{1}(w), e_{1})\\
 & \equiv \Ua_{\mathsmaller{\Aii}}(\pr_{1}(z), \pr_{1}(w), \Idtoeqv_{\mathsmaller{\Aii}}(\pr_{1}(z), \pr_{1}(w), p_{1}))\\
 & = p_{1},
\end{align*}
and proceeding similarly we get $p_{2}{'} = p_{2}$. With the use of $\ap_{\paire}$ and the 
property of inversion between $\paire$ and $\epair$, shown in section 2.6 
of~\cite{HoTT13}, we get
\begin{align*}
\Ua_{\mathsmaller{\Aii \times \Bii}}(z, w, e{'}) & = \paire(z, w, p_{1}, p_{2})\\
& \equiv \paire(z, w, \pr_{1}(\epair(z, w, p)), \pr_{2}(\epair(z, w, p)))\\
& = p.
\end{align*}
Using the definitions in Lemma~\ref{lem: idtoeqvprod} and the functions $\ap_{\pr_{1}}$ and $\Idtoeqv_{\mathsmaller{\Aii}}^{2}$,
if $p \equiv \Ua_{\mathsmaller{\Aii \times \Bii}}(z, w, e{'}) \equiv \paire(z, w, \pr_{1}{'}, \pr_{2}{'})$, then
$$p_{1} \equiv \pr_{1}(\epair(\paire(z, w, p_{1}{'}, p_{2}{'}))) = \pr_{1}((p_{1}{'}, p_{2}{'})) = p_{1}{'},$$
$$e_{1} \cong_{\mathsmaller{\Aii}} \Idtoeqv_{\mathsmaller{\Aii}}(\pr_{1}(z), \pr_{1}(w), p_{1}{'}) \equiv \Idtoeqv_{\mathsmaller{\Aii}}(\pr_{1}(z), \pr_{1}(w), \Ua_{\mathsmaller{\Aii}}(\pr_{1}(z), \pr_{1}(w), e_{1}{'}))\cong_{\mathsmaller{\Aii}}  e_{1}{'},$$
and proceeding similarly we get $p_{2} = p_{2}{'}$ and $e_{2}= e_{1}{'}$.
By Corollaries~\ref{crl: T} and~\ref{crl: T2} we get
$$ \Idtoeqv_{\mathsmaller{\Aii \times \Bii}}(z, w, \Ua_{\mathsmaller{\Aii \times \Bii}}(z, w, e{'})) 
\equiv T(z, w, e_{1}, e_{2}) \cong_{\mathsmaller{\Aii \times \Bii}} T(z, w, e_{1}{'}, e_{2}{'}) \cong_{\mathsmaller{\Aii \times \Bii}} e{'}.$$
To define $\Ua_{\mathsmaller{\Aii \times \Bii}}^{2}$, let $e{'}, d{'}: 
z \simeq_{\mathsmaller{\Aii \times \Bii}} w$, $i{'} : e{'}
\cong_{\mathsmaller{\Aii \times \Bii}} d{'} \equiv (e_{1}{'} \cong_{\mathsmaller{\Aii}} d_{1}{'})
\times (e_{2}{'} \cong_{\mathsmaller{\Bii}} d_{2}{'})$, and
$i_{1}{'} \equiv \pr_{1}(i{'}), i_{2}{'} \equiv \pr_{2}(i{'})$. As $\Ua_{\mathsmaller{\Aii}}^{2}(\pr_{1}(z), \pr_{1}(w), 
e_{1}{'}, d_{1}{'}, i_{1}{'}) : p_{1}{'} = q_{1}{'}$ and  $\Ua_{\mathsmaller{\Bii}}^{2}(\pr_{2}(z), \pr_{2}(w), 
e_{2}{'}, d_{2}{'}, i_{2}{'}) : p_{2}{'} = q_{2}{'}$, we define
$$\Ua_{\mathsmaller{\Aii \times \Bii}}^{2}(z, w, e{'}, d{'}, i{'}) \equiv \ap_{\paire(z, w)}\Bigg((p_{1}{'}, p_{2}{'}), 
(q_{1}{'}, q_{2}{'}), \Ua_{\mathsmaller{\Aii}}^{2}(\pr_{1}(z), \pr_{1}(w), 
e_{1}{'}, d_{1}{'}, i_{1}{'}),$$
$$ \ \ \ \ \ \ \ \ \ \ \ \  \Ua_{\mathsmaller{\Bii}}^{2}(\pr_{2}(z), \pr_{2}(w), 
e_{2}{'}, d_{2}{'}, i_{2}{'})\Bigg),$$
where 
$$\ap_{\paire(z, w)}((p_{1}{'}, p_{2}{'}), 
(q_{1}{'}, q_{2}{'})) : (p_{1}{'} = q_{1}{'}) \times (p_{2}{'} = q_{2}{'}) \rightarrow \paire(z, w, p_{1}{'}, p_{2}{'}) =
\paire(z, w, q_{1}{'}, q_{2}{'}).$$
\end{proof}

\begin{proposition}\label{prp: produt}
Let $\Aii \times \Bii$ be univalent. If $(B, b)$ is a pointed type i.e., $b : B$, then $\Aii$ is univalent. If $(A, a)$ is a pointed type, then $\Bii$ is univalent.
\end{proposition}

\begin{proof}
We only sketch the main idea of the proof for the first case.
If $e : x \simeq_{\mathsmaller{\Aii}} x{'}$, then $(e, \eqv_b) : (x, b) \simeq_{\mathsmaller{\Aii \times \Bii}} (x{'}, b)$. The resulting term of type $(x, b) =_{A \times B} (x{'}, b)$ generates a term of type $x =_A x{'}$. All details are handled as in the proof of Theorem~\ref{thm: utprod}. 
\end{proof}

\section{Exponential typoid}
\label{sec:exp}

If $\Aii, \Bii$ are typoids and $\fAB$, the type ``$f$ is a typoid function from $\Aii$ to $\Bii$'' is
$$\Typfun(f) \equiv \sum_{\Phi_{f}: \prod_{x, y :A}\prod_{e: x \simeq_{\mathsmaller{\Aii}} y}f(x) 
\simeq_{\mathsmaller{\Bii}} f(y)}\Bigg[\Bigg(\prod_{x, y :A}\prod_{e: x \simeq_{\mathsmaller{\Aii}} y}\prod_{d: y \simeq_{\mathsmaller{\Aii}} z}$$
$$\Big(\Phi_{f}(x, x, \eqv_{x}) \cong_{\mathsmaller{\Bii}} \eqv_{f(x)}\Big) \times 
\Big(\Phi_{f}(x, z, e \ast_{\mathsmaller{\Aii}} d) \cong_{\mathsmaller{\Bii}} \Phi_{f}(x, y, e) \ast_{\mathsmaller{\Bii}} \Phi_{f}(y, z, d)\Big)\Bigg) \times$$
$$\times \Bigg(\prod_{x, y :A}\prod_{e, d: x \simeq_{\mathsmaller{\Aii}} y}\prod_{i: e \cong_{\mathsmaller{\Aii}} d}\Phi_{f}(x, y, e)
\cong_{\mathsmaller{\Bii}} \Phi_{f}(x, y, d)\Bigg)\Bigg].$$
A canonical term of type $\Typfun(f)$ is a pair $(\Phi_{f}, (U, \Phi_{f}^{2}))$, or simply a triplet 
$(\Phi_{f}, U, \Phi_{f}^{2})$, where $U$ is a term of the first type 
of the outer product and $\Phi_{f}^{2}$ is a term of the second. The \textit{exponential}  of $\Aii, \Bii$ is the typoid 
$\Bii^{\Aii} := \big(B^A, \simeq_{\mathsmaller{\Bii^{\Aii}}}, \eqv_{\mathsmaller{\Bii^{\Aii}}}, \ast_{\mathsmaller{\Bii^{\Aii}}}, ^{-1_{\mathsmaller{\Bii^{\Aii}}}}, \cong_{\mathsmaller{\Bii^{\Aii}}}\big)$, where
$$B^{A} \equiv \sum_{\fAB}\Typfun(f).$$
If $\phi \equiv (f, \Phi_{f}, U, \Phi_{f}^{2})$ and $\theta \equiv (g, \Phi_{g}, W, \Phi_{g}^{2})$ are two canonical
terms of type $B^{A}$, we define
$$\phi \simeq_{\mathsmaller{\Bii^{\Aii}}} \theta \equiv \sum_{\Theta_{f,g}: \prod_{x : A}f(x) \simeq_{\mathsmaller{\Bii}} g(x)}\Bigg(\prod_{x, y : A} \prod_{e: x \simeq_{\mathsmaller{\Aii}} y}
\Phi_{f(x, y, e)} \ast_{\mathsmaller{\Bii}} \Theta_{f, g}(y) \cong_{\mathsmaller{\Bii}} \Theta_{f,g}(x) \ast_{\mathsmaller{\Bii}} \Phi_{g}(x, y, e)\Bigg).$$
A canonical term $e$ of type $\phi \simeq_{\mathsmaller{\Bii^{\Aii}}} \theta$ is a pair $(\Theta_{f,g}, \Theta_{f,g}^{2})$, where
$$\Theta_{f,g}^{2} : \prod_{x, y : A}\prod_{e: x \simeq_{\mathsmaller{\Aii}} y}
\Phi_{f(x, y, e)} \ast_{\mathsmaller{\Bii}} \Theta_{f,g}(y) \cong_{\mathsmaller{\Bii}} \Theta_{f,g}(x) \ast_{\mathsmaller{\Bii}} \Phi_{g}(x, y, e)$$
i.e., $\Theta_{f,g}^{2}(x, y, e)$ is a proof that the following diagram commutes
\begin{center}
\begin{tikzpicture}
\node (E) at (0,0) {$g(x)$};
\node[right=of E] (G) {};
\node[right=of G] (F) {$g(y).$};
\node[above=of F] (A) {$f(y)$};
\node [above=of E] (D) {$f(x)$};
\draw[double distance = 1.5pt] (E)--(F) node [midway,below]{$\Phi_g(x,y,e)$};
\draw[double distance = 1.5pt] (D)--(A) node [midway,above] {$\Phi_f(x,y,e)$};
\draw[double distance = 1.5pt] (D)--(E) node [midway,left] {$\Theta_{f,g}(x)$};
\draw[double distance = 1.5pt] (A)--(F) node [midway,right] {$\Theta_{f,g}(y)$};
\end{tikzpicture}
\end{center}
If $\phi$ is a canonical term of type $B^{A}$ we define $\eqv_{\phi}: \phi \simeq_{\mathsmaller{\Bii^{\Aii}}} \phi$ as the pair
$(\Theta_{f,f}, \Theta_{f,f}^{2})$, where $\Theta_{f,f} \equiv \lambda(x:A).\eqv_{f(x)} : \prod_{x:A}f(x) \simeq_{\Bii} f(x)$
and $\Theta_{f,f}^{2}(x, y, e)$ is the obvious proof that the following diagram commutes
\begin{center}
\begin{tikzpicture}

\node (E) at (0,0) {$g(x)$};
\node[right=of E] (G) {};
\node[right=of G] (F) {$g(y).$};
\node[above=of F] (A) {$f(y)$};
\node [above=of E] (D) {$f(x)$};
\draw[double distance = 1.5pt] (E)--(F) node [midway,below]{$\Phi_g(x,y,e)$};
\draw[double distance = 1.5pt] (D)--(A) node [midway,above] {$\Phi_f(x,y,e)$};
\draw[double distance = 1.5pt] (D)--(E) node [midway,left] {$\eqv_{f(x)}$};
\draw[double distance = 1.5pt] (A)--(F) node [midway,right] {$\eqv_{f(y)}$};

\end{tikzpicture}
\end{center}
If $\phi \equiv (f, \Phi_{f}, U, \Phi_{f}^{2}), \theta \equiv (g, \Phi_{g}, W, \Phi_{g}^{2}), 
\eta \equiv (h, \Phi_{h}, V, \Phi_{h}^{2})$ are canonical terms of type $B^{A}$ and $e \equiv 
(\Theta_{f,g}, \Theta_{f,g}^{2}) : \phi \simeq_{\mathsmaller{\Bii^{\Aii}}} \theta$ and 
$d \equiv (\Theta_{g,h}, \Theta_{g,h}^{2}) : \theta \simeq_{\mathsmaller{\Bii^{\Aii}}} \eta$, we define 
$$e \ast_{\mathsmaller{\Bii^{\Aii}}} d \equiv (\Theta_{f,h}, \Theta_{f,h}^{2}) : \phi \simeq_{\mathsmaller{\Bii^{\Aii}}} \eta$$
$$\Theta_{f,h} \equiv \lambda(x:A). \Theta_{f,g}(x) \ast_{\mathsmaller{\Bii}} \Theta_{g,h}(x).$$
A term $\Theta_{f,h}^{2}(x, y, e)$ of type 
$$\Phi_{f(x, y, e)} \ast_{\mathsmaller{\Bii}} \Theta_{f,h}(y) \cong_{\mathsmaller{\Bii}} \Theta_{f,h}(x) \ast_{\mathsmaller{\Bii}} \Phi_{h}(x, y, e) \equiv$$
$$\Phi_{f(x, y, e)} \ast_{\mathsmaller{\Bii}} (\Theta_{f,g}(y) \ast_{\mathsmaller{\Bii}} \Theta_{g,h}(y)) \cong_{\mathsmaller{\Bii}} (\Theta_{f,h}(x) \ast_{\mathsmaller{\Bii}} 
\Theta_{g,h}(x)) \ast_{\mathsmaller{\Bii}} \Phi_{h}(x, y, e) $$
is found through the commutativity of the following outer diagram
\begin{center}
\begin{tikzpicture}
\node (E) at (0,0) {$g(x)$};
\node[right=of E] (G) {};
\node[right=of G] (F) {$g(y)$};
\node[above=of F] (A) {$f(y)$};
\node [above=of E] (D) {$f(x)$};
\node [below=of E] (K) {$h(x)$};
\node [below=of F] (L) {$h(y),$};
\draw[double distance = 1.5pt] (E)--(F) node [midway,below]{$\Phi_g(x,y,e)$};
\draw[double distance = 1.5pt] (D)--(A) node [midway,above] {$\Phi_f(x,y,e)$};
\draw[double distance = 1.5pt] (D)--(E) node [midway,left] {$\Theta_{f,g}(x)$};
\draw[double distance = 1.5pt] (A)--(F) node [midway,right] {$\Theta_{f,g}(y)$};
\draw[double distance = 1.5pt] (K)--(L) node [midway,below] {$\Phi_h(x,y,e)$};
\draw[double distance = 1.5pt] (E)--(K) node [midway,left] {$\Theta_{g,h}(x)$};
\draw[double distance = 1.5pt] (F)--(L) node [midway,right] {$\Theta_{g,h}(y)$};
\end{tikzpicture}
\end{center}
and rests on the commutativity of the inner diagrams. If
$e \equiv (\Theta_{f,g}, \Theta_{f,g}^{2}) : \phi \simeq_{\mathsmaller{\Bii^{\Aii}}} \theta$, let
$$e^{-1_{\mathsmaller{\Bii^{\Aii}}}} \equiv (\Theta_{f,g}^{-1}, [\Theta_{f,g}^{2}]^{-1}) : \theta \simeq_{\mathsmaller{\Bii^{\Aii}}} \phi,$$
where $\Theta_{f,g}^{-1} : \prod_{x : A}g(x) \simeq_{\mathsmaller{\Bii}} f(x)$ is defined by
$$\Theta_{f,g}^{-1}(x) \equiv [\Theta_{f,g}(x)]^{-1_{\mathsmaller{\Bii}}},$$
for every $x : A$, and $[\Theta_{f,g}^{2}]^{-1}(y, x, e)$ is a term of type 
$$\Phi_{g}(y, x, e) \ast_{\mathsmaller{\Bii}} \Theta_{f,g}(x)^{-1} \cong_{\mathsmaller{\Bii}} 
\Theta_{f,g}(y)^{-1} \ast_{\mathsmaller{\Bii}} \Phi_{f}(y, x, e)$$
i.e., a proof of the commutativity of the following diagram
\begin{center}
\begin{tikzpicture}
\node (E) at (0,0) {$f(x)$};
\node[right=of E] (G) {};
\node[right=of G] (F) {$f(y),$};
\node[above=of F] (A) {$g(y)$};
\node [above=of E] (D) {$g(x)$};
\draw[double distance = 1.5pt] (E)--(F) node [midway,below]{$\Phi_f(y,x,e)$};
\draw[double distance = 1.5pt] (D)--(A) node [midway,above] {$\Phi_g(y,x,e)$};
\draw[double distance = 1.5pt] (D)--(E) node [midway,left] {$\Theta_{f,g}(x)^{-1}$};
\draw[double distance = 1.5pt] (A)--(F) node [midway,right] {$\Theta_{f,g}(y)^{-1}$};
\end{tikzpicture}
\end{center}
which rests on the commutativity of the diagram that corresponds to the term 
$\Theta_{f,g}^{2}(x, y, e^{-1})$.

\begin{proposition}\label{prp: exp}
Let $\Aii, \Bii$ be typoids. If $\Bii$ is univalent, then $\Bii^{\Aii}$ is univalent.
\end{proposition}

\begin{proof}
We only sketch the main idea of the proof.
If $\phi \equiv (f, \Phi_{f}, U, \Phi_{f}^{2})$ and $\theta \equiv (g, \Phi_{g}, W, \Phi_{g}^{2})$ are two canonical
terms of type $B^{A}$, and if $(\Theta_{f,g}, \Theta_{f,g}^{2})$ is a canonical term of type $\phi \simeq_{\mathsmaller{\Bii^{\Aii}}} \theta$, then $\Theta_{f,g}: \prod_{x : A}f(x) \simeq_{\mathsmaller{\Bii}} g(x)$.
Since $\Bii$ is univalent, every term $\Theta_{f,g}(x) : f(x) \simeq_{\mathsmaller{\Bii}} g(x)$ generates a term of type $f(x) =_B g(x)$. Consequently, we get a term of type $f =_{A \to B} g$. Using Theorem 2.7.2 in~\cite{HoTT13}, we can construct a term of type $\phi = \theta$. 
\end{proof}

\section{Truncated typoids}
\label{sec: trunc}

In~\cite{HoTT13}, section 3.7, the notion of propositional truncation of a type $A$ is
implemented through the higher inductive type $||A||$. 
Here we use the ``truncated typoid'' $\At$ to interpret this notion. Starting from a typoid structure on a type $A$ we define a new typoid stucture on $A$, which behaves accordingly. Hence, we keep the same type and we change
the typoid structure, while in the theory of HITs the type is changed. We denote the unit type by $\one$.

\begin{definition}\label{def: trt}
If $A : \Uii$, we call the typoid 
 $\Aii^{t} \equiv \big(A, \simeq_{\mathsmaller{\Aii^{t}}}, \eqv_{\mathsmaller{\Aii^{t}}}, \ast_{\mathsmaller{\Aii^{t}}}, ^{-1_{\mathsmaller{\Aii^{t}}}}, \cong_{\mathsmaller{\Aii^{t}}}\big)$
\textit{truncated}, where for every $x, y, z :A$, $e, e{'} : x \simeq_{\mathsmaller{\Aii^{t}}} y$, and $d : y \simeq_{\mathsmaller{\Aii^{t}}} z$, let
$$x \simeq_{\mathsmaller{\Aii^{t}}} y \equiv \one, \ \ \ \eqv_{\mathsmaller{\Aii^{t}}}(x) \equiv 0_{\one}, \ \ \
\ast_{\mathsmaller{\Aii^{t}}}(x, y, z, e, d) \equiv 0_{\one}, $$
$$^{-1_{\mathsmaller{\Aii^{t}}}}(x, y, e) \equiv 0_{\one}, \ \ \ \cong_{\mathsmaller{\Aii^{t}}}(x, y, e, e{'}) \equiv (e = e{'}).$$

\end{definition}

The proof that $\Aii^{t}$ is a typoid is immediate. One needs only to take into account that 
$\isProp(\one)$, where $\isProp(A) \equiv \prod_{x, y : A}(x =_{A} y)$.

\begin{proposition}\label{tr1}
 If $A : \Uii$, $\Bii$ is a typoid and $f: B \rightarrow A$, then $f$ is a typoid function from $\Bii$ to $\Aii^{t}$.
\end{proposition}

\begin{proof}Let $x, y, z : B$, $e, e{'} : x \simeq_{\mathsmaller{\Bii}} y$ $i: e \cong_{\mathsmaller{\Bii}} e{'}$, 
and $d : y \simeq_{\mathsmaller{\Bii}} z$ We define $\Phi_{f}(x, y, e) \equiv 0_{\one}$, hence $\Phi_{f}(x, y, e) :
f(x) \simeq_{\mathsmaller{\Aii^{t}}} f(y)$, and we also define $\Phi_{f}^{2}(x, y, e, e{'}, i) 
\equiv \refl_{0_{\one}}$, hence $\Phi_{f}^{2}(x, y, e, e{'}, i) : 0_{\one} =_{\one} 0_{\one}$. Clearly, 
$\Phi_{f}(x, x, \eqv_{x}) \equiv 0_{\one} \equiv e_{f(x)},$
and $\Phi_{f}(x, z, e_{1} \ast_{\mathsmaller{\Bii}} e_{2}) \equiv 0_{\one} = 
(0_{\one} \ast_{\mathsmaller{\Aii^{t}}} 0_{\one}) \equiv \Phi_{f}(x, y, e_{1}) 
\ast_{\mathsmaller{\Aii^{t}}} \Phi_{f}(y, z, e_{2}).$
\end{proof}

\begin{corollary}\label{crl: tr2}
 If $A, B : \Uii$ and $f: B \rightarrow A$, then $f$ is a typoid function from $\Bii^{t}$ to $\Aii^{t}$. 
\end{corollary}

If $f: B \rightarrow A$, we use the notation $f^{t}$ for $f$ to indicate that we view $f$ as a typoid function
from $B^{t}$ to $A^{t}$. In the next proof we use the type $\isSet(A) \equiv \prod_{\xyA}\prod_{p, q : \Exy}(p = q)$.

\begin{proposition}\label{prp: tr3}
 If $A : \Uii$ such that $\isProp(A)$, then $\Aii^{t}$ is univalent.
\end{proposition}

\begin{proof}If $\Omega : \isProp(A)$, $d, e : \one$ and 
$i : d \cong_{\mathsmaller{\Aii^{t}}} e \equiv (d = e)$, we define
$$\Ua_{\mathsmaller{\Aii^{t}}}(x, y, d) \equiv \Omega(x, y), \ \ \ \ \Ua_{\mathsmaller{\Aii^{t}}}^{2}(x, y, d, e, i) 
\equiv \refl_{\Omega(x, y)},$$
hence $\Ua_{\mathsmaller{\Aii^{t}}}(x, y, d) : x =_{A} y$ and $\Ua_{\mathsmaller{\Aii^{t}}}^{2}(x, y, d, e, i)$
is a term of type 
$\Omega(x, y) = \Omega(x, y) \equiv \Ua_{\mathsmaller{\Aii^{t}}}(x, y, d) = \Ua_{\mathsmaller{\Aii^{t}}}(x, y, e)$.
Let $\Idtoeqv_{\mathsmaller{\Aii^{t}}}$ be an
$1$-associate of $\id_{A}$ seen as function from $\Aii_{0}$ to $\Aii^{t}$. As
$$\Idtoeqv_{\mathsmaller{\Aii^{t}}}(x, y, \Ua_{\mathsmaller{\Aii^{t}}}(x, y, d)) : \one,$$ and $d : \one$,
we get 
$\Idtoeqv_{\mathsmaller{\Aii^{t}}}(x, y, \Ua_{\mathsmaller{\Aii^{t}}}(x, y, d)) = d,$ i.e., 
$\Idtoeqv_{\mathsmaller{\Aii^{t}}}(x, y, \Ua_{\mathsmaller{\Aii^{t}}}(x, y, d)) \cong_{\mathsmaller{\Aii^{t}}} d.$
Since $\Ua_{\mathsmaller{\Aii^{t}}}(x, y, \Idtoeqv_{\mathsmaller{\Aii^{t}}}(x, y, p)) : x =_{A} y,$
$p : x =_{A} y$ and $\isProp(A) \rightarrow \isSet(A)$ (see Lemma 3.3.4 of~\cite{HoTT13}), 
we conclude that $\Ua_{\mathsmaller{\Aii^{t}}}(x, y, \Idtoeqv_{\mathsmaller{\Aii^{t}}}(x, y, p)) = p.$
\end{proof}

Using Theorem~\ref{thm: ut2}(i) we get the following corollary.

\begin{corollary}\label{crl: tr4}
 If $A : \Uii$ such that $\isProp(A)$, $\Bii$ is a typoid and $\fAB$, then $f$ is a typoid function 
 from $\Aii^{t}$ to $\Bii$.  
\end{corollary}

\begin{proposition}\label{prp: tr5}
 If $A : \Uii$, $\Bii$ is a typoid such that $\isProp(B)$, and $\fAB$, 
 then $f$ is a typoid function from $\Aii^{t}$ to $\Bii$.
 
\end{proposition}

\begin{proof} By Corollary~\ref{crl: tr2} we have that $f$ is a typoid function from $\Aii^{t}$ to
 $\Bii^{t}$, while by Corollary~\ref{crl: tr4} $\id_{B}$ is a typoid function from $\Bii^{t}$ to $\Bii$. By 
 Proposition~\ref{prp: comp} we conclude that $f \equiv \id_{B} \circ f$ is a typoid function from $\Aii^{t}$ to $\Bii$. 
\end{proof}

\section{Concluding remarks and future work}
\label{sec: concl}

We presented here some first results on the fundamental properties of general typoids and univalent typoids. 
Our main goal was to establish a common framework for all instances of types that behave in a ``univalent way''. The structure of typoid is a weak groupoid analogue to the notion of precategory in HoTT, and the notion of a univalent typoid is a weak groupoid analogue to the notion of category in HoTT\footnote{See Chapter 9 in book-HoTT~\cite{HoTT13}. The notion of category in HoTT was first considered in~\cite{HS98}, and it was formalised in~\cite{AKS13}.}. 

A precategory in HoTT is a type $A$ such that for every $a, b : A$ there is a ``set'' $\hom(a, b)$ satisfying the expected conditions of a category. If $a \cong b$ is the type of the corresponding notion of isomorphism, a precategory is a category, if the obvious function that sends a term of $a =_A b$ to a term of type $a \cong b$ is an equivalence. 
The analogy between typoids/univalent typoids and precategories/categories suggests naturally the study of topics, like the Rezk completion i.e., the universal way to replace a precategory by a category, within the context of typoids\footnote{We would like to thank Steve Awodey for pointing the Rezk completion to us as a topic of possible study within the theory of typoids.}. 

We plan to study the category of typoids and other categorical properties of univalent typoids in a subsequent work.
More specific examples of univalent typoids are expected to be found in~\cite{CD13}, where with the use of UA a
general theorem that ``isomorphism implies equality" is shown for many algebraic structures. The implementation of 
more higher inductive types as appropriate typoids is another topic of future work.

\vspace{8mm}
\noindent
\textbf{Acknowledgment}\\
This work was completed during our visit to CMU and was supported by the EU-project
``Computing with Infinite Data''. We would like to thank Wilfried Sieg for hosting us in CMU.


\begin{thebibliography}{10}\label{bibliography}

\bibitem{AKS13} B. Ahrens, K. Kapulkin, M. Shulman: Univalent categories and the
Rezk completion, arXiv:1303.0584, 2013.
\bibitem{AW09} S. Awodey, M. A. Warren. Homotopy theoretic models of identity types, Mathematical
Proceedings of the Cambridge Philosophical Society, 2009, 146:45-55.
\bibitem{BC03} G. Barthe, V. Capretta: Setoids in type theory, JFP 13 (2): 261-293, 2003.
\bibitem{Bi67} E. Bishop: \textit{Foundations of Constructive Analysis}, McGraw-Hill, 1967.
\bibitem{CS16} J. Carette, A. Sabry: Computing with semirings and weak rig groupoids, ESOP 2016 : \textit{Programming Languages and Systems}, 123-148.
\bibitem{CCCS18} J. Carette, C. Chen, V. Choudhury, A. Sabry: From reversible programs to univalent universes and back, El. Notes in Theor. Comp. Sci. 336, 2018, 5-25.
\bibitem{CS07} T. Coquand, A. Spiwack: Towards Constructive Homological Algebra in Type Theory, 
in M. Kauers et.al (Eds.) Towards Mechanized Mathematical Assistants,
Lecture Notes in Computer Science, vol 4573. Springer, Berlin, Heidelberg, 2007, 40-54.
\bibitem{CD13} T. Coquand, A. Danielsson: Isomorphism is equality, Indagationes Mathematicae 24 (4), 2013, 1105-1120.
\bibitem{HS98} M. Hofmann, T. Streicher: The groupoid interpretation of type theory, G. Sambin, J. M. Smith (Eds.)
\textit{Twenty-five years of constructive type theory} (Venice, 1995), volume 36 of Oxford Logic
Guides, Oxford University Press, 1998, 83-111.
\bibitem{ML75} P. Martin-L\"{o}f: An intuitionistic theory of types: predicative part, in H. E. Rose and
J. C. Shepherdson (Eds.) \textit{Logic Colloquium'73}, pp.73-118, North-Holland, 1975.
\bibitem{ML98} P. Martin-L\"{o}f: An intuitionistic theory of types, in G. Sambin, J. M. Smith (Eds.)
\textit{Twenty-five years of constructive type theory} (Venice, 1995), volume 36 of Oxford Logic
Guides, Oxford University Press, 1998, 127-172.
\bibitem{Pa05} E. Palmgren: Bishop's set theory, TYPES summer school G\"{o}teborg, August 2005.
\bibitem{Pe19} I.~Petrakis: Dependent Sums and Dependent Products in Bishop's Set Theory,
in P. Dybjer et. al. (Eds) TYPES 2018, LIPIcs, Vol. 130, Article No. 3, 2019.
\bibitem{Pe20} I.~Petrakis: \textit{Families of Sets in Bishop Set Theory}, Habilitationsschrift, LMU, Munich, 2020,
available at \texttt{https://www.mathematik.uni-muenchen.de/$\sim$petrakis/content/Theses.php}.
 \bibitem{HoTT13} The Univalent Foundations Program: \textit{Homotopy Type Theory: Univalent Foundations of Mathematics}, 
Institute for Advanced Study, Princeton, 2013. 
\bibitem{Vo06} V. Voevodsky. A very short note on the homotopy $\lambda$-calculus, in 
http://www.math.\\
ias.edu/$\sim$vladimir/Site3/Univalent$_{-}$Foundations$_{-}$files/Hlambda$_{-}$short$_{-}$current.\\
pdf, 2006.


\end{thebibliography}
\end{document}